\documentclass[11pt,a4paper]{article}
\usepackage{amsfonts,amsgen,amstext,amsbsy,amsopn,amsfonts,amssymb,amscd}
\usepackage[leqno]{amsmath}
\usepackage[amsmath,amsthm,thmmarks]{ntheorem}
\usepackage{epsf,epsfig}
\usepackage{float}
\usepackage{dsfont}
\usepackage{ebezier,eepic}
\usepackage{color}
\usepackage{tikz}
\usepackage{multirow}
\usepackage{mathrsfs}
\usepackage{graphicx}
\usepackage{subfigure}
\newtheorem{thm}{Theorem}[section]

\newtheorem{prop}[thm]{Proposition}
\newtheorem{prob}[thm]{Problem}

\newtheorem{lem}[thm]{Lemma}
\newtheorem{example}[thm]{Example}
\newtheorem{false statement}{False statement}
\newtheorem{cor}[thm]{Corollary}

\theoremstyle{definition}

\newtheorem{claim}[thm]{Claim}

\makeatletter \@addtoreset{equation}{section}

\def\hh{\mathcal{H}}

\def\hl{\mathcal{L}}
\def\hht{\mathcal{T}}

\def\hf{\mathcal{F}}
\def\hg{\mathcal{G}}
\def\hk{\mathcal{K}}
\def\ha{\mathcal{A}}
\def\hb{\mathcal{B}}
\def\hs{\mathcal{S}}
\def\hr{\mathcal{R}}

\def\hi{\mathcal{I}}

\def\hp{\mathcal{P}}

\begin{document}

\title{\bf\Large Intersecting families with covering number three II}
\date{}
\author{Peter Frankl$^1$, Jian Wang$^2$\\[10pt]
$^{1}$R\'{e}nyi Institute, Budapest, Hungary\\[6pt]
$^{2}$Department of Mathematics, Sichuan University, Chengdu, 610065, China.\\[6pt]
E-mail:  $^1$frankl.peter@renyi.hu, $^2$wangjianmath01@scu.edu.cn
}

\maketitle

\begin{abstract}
A family $\hf\subset \binom{[n]}{k}$ is called {\it intersecting} if $F\cap F'\neq \emptyset$ for all $F,F'\in \hf$. The {\it covering number} of a family $\hf$ is defined as the minimum size of $T\subset [n]$ such that $T\cap F\neq \emptyset$ for all $F\in \hf$. In 1980, the first author proved that for sufficiently large $n$, any intersecting $k$-graph $\mathcal{F}$ with covering number at least three, satisfies $|\mathcal{F}|\leq \binom{n-1}{k-1}-\binom{n-k}{k-1}-\binom{n-k-1}{k-1}+\binom{n-2k}{k-1}+\binom{n-k-2}{k-3}+3$.
There was very little progress during more than forty years but recently (cf. \cite{FW25}) with a completely
 different approach we proved the same result for the full range  $n\geq 2k$ and $k\geq 7$.  In this short paper we prove the same inequality for all the remaining cases.
\end{abstract}

\section{Introduction}

Let $[n]$ be the standard $n$-element set $\{1,2,\ldots,n\}$ and let $\binom{[n]}{k}$ denote the family of all $k$-element subsets of $[n]$. A family $\hf\subset \binom{[n]}{k}$ is called {\it intersecting} if $F\cap F'\neq \emptyset$ for all $F,F'\in \hf$.

One of the most important results in extremal set theory is the Erd\H{o}s-Ko-Rado Theorem.

\begin{thm}[Erd\H{o}s-Ko-Rado \cite{ekr}]
Let $\hf\subset \binom{[n]}{k}$ be an intersecting family with $n\geq 2k$. Then
\begin{align}\label{ineq-ekr}
|\hf|\leq \binom{n-1}{k-1}.
\end{align}
Moreover, for $n>2k$ the equality holds if and only if $\hf$ is a full star, that is, $\hf=\{F\in\binom{[n]}{k}\colon x\in F\}$ for some $x\in [n]$.
\end{thm}

For $\hf\subset \binom{[n]}{k}$ and $T\subset [n]$,
we call $T$ a {\it cover} of $\hf$ if $F\cap T\neq \emptyset$ for all $F\in \hf$. The {\it covering number} $\tau(\hf)$ is defined as the minimum size of a cover of $\hf$. We call $\hf$ a star if $\tau(\hf)=1$.

\begin{thm}[Hilton-Milner \cite{HM67}]
Suppose that $n> 2k\geq 4$, $\hf\subset\binom{[n]}{k}$ is intersecting and $\tau(\hf)\geq 2$, then
\begin{align}\label{ineq-hm}
|\hf| \leq \binom{n-1}{k-1}-\binom{n-k-1}{k-1}+1.
\end{align}
\end{thm}

Since in an intersecting family $\hf\subset \binom{[n]}{k}$ every edge is a cover, $\tau(\hf)\leq k$ is obvious. In their seminal paper \cite{EL} among other things Erd\H{o}s and Lov\'{a}sz examined the maximal size $m(k)$ of an intersecting family of $k$-sets with covering number $k$. They proved
\begin{align}\label{ineq-1.3}
\lfloor k!(e-1)\rfloor \leq m(k) \leq k^k.
\end{align}

Both the lower and the upper bounds have been improved throughout the years (cf. \cite{F19}, \cite{Z}) but determining $m(k)$ exactly appears to be extremely difficult. The known values are $m(2)=3$ (trivial), $m(3)=10$ \cite{L}. For $k=4$ already there is a large gap:
\begin{align}
42 \leq r(4) \leq 64,
\end{align}
where the upper bound was proved very recently (cf \cite{FW25-2}).

In view of \eqref{ineq-hm} and \eqref{ineq-1.3} it is natural to consider the following general problem.

\begin{prob}
Let $1\leq r\leq k$, $n\geq 2k$. Determine or estimate
\[
m(n,k,r):=\max\left\{|\hf|\colon \hf\subset \binom{[n]}{k},\ \hf \mbox{ is intersecting with }\tau(\hf)\geq r\right\}.
\]
\end{prob}

The Erd\H{o}s-Ko-Rado and Hilton-Milner Theorems can be stated as $m(n,k,1)=\binom{n-1}{k-1}$, $m(n,k,2)=\binom{n-1}{k-1}-\binom{n-k-1}{k-1}+1$, respectively.

For a family $\hh\subset \binom{[n]}{k}$ and a positive integer $\ell$, define
\[
\hht(\hh)=\left\{T\subset [n]\colon |T|\leq k,\  T \mbox{ is a cover of }\hh\right\},\ \hht^{(\ell)}(\hh) =\hht(\hh) \cap \binom{[n]}{\ell}.
\]

We say that an intersecting family $\hf\subset \binom{[n]}{k}$ is {\it saturated} if any addition of an extra $k$-set would destroy the intersecting property.

\begin{prop}\label{prop-1.4}
If $n\geq 2k$ and $\hh\subset \binom{[n]}{k}$ is saturated  intersecting,  then $\hht(\hh)$ is intersecting as well.
\end{prop}
\begin{proof}
If $T,T' \in \hht(\hh)$ are disjoint then using $n\geq 2k$ we can find $F,F' \in {[n] \choose k}$ with $T\subset F$,
$T'\subset F'$ and $F \cap F' = \emptyset$.  By saturatedness both $F$ and $F'$ are in $\hh$, a contradiction.
\end{proof}

Let us describe a general construction of relatively large intersecting families with covering number $r$, $r\geq 3$.

\begin{example}
Let $\hh\subset \binom{[2,n]}{k}$ be intersecting and $\tau(\hh)= r-1$. Define
\[
\hf_{\hh} =\hh \cup \left\{F\in \binom{[n]}{k}\colon 1\in F,\ \exists\ T\in \hht(\hh),\ T\subset F\right\}.
\]
It should be clear that $\hf_\hh$ is intersecting and $\tau(\hf_\hh)\leq r$.
\end{example}

\begin{prop}
If $n\geq 2k$ and $\tau(\hht(\hh)\setminus \hht^{(k)}(\hh))\geq r$, then $\tau(\hf_\hh)=r$.
\end{prop}

\begin{proof}
Suppose for contradiction that $S$ is a cover of $\hf_\hh$ with $|S|=r-1$. Since $\hh\subset \hf_\hh$ and $\tau(\hh)=r-1$,
$S\subset [2,n]$ and $S\in \hht^{(r-1)}(\hh)$ follow.

On the other hand $\tau(\hht(\hh)\setminus \hht^{(k)}(\hh))\geq r$ implies the existence of $T\in \hht(\hh)\setminus \hht^{(k)}(\hh)$ with $T\cap S=\emptyset$. As $|\{1\}\cup T|\leq k$ there exists $F\in \hf$ with $F\cap S=\emptyset$, a contradiction.
\end{proof}

Let $[a,b]=\{i\colon a\leq i\leq b\}$ be the discrete interval.

\begin{example}
Define
\[
\hb = \{[2,k+1], \{2\}\cup [k+2,2k], \{3\}\cup [k+2,2k]\}
\]
and
\[
\ha =\left\{A\in \binom{[n]}{k}\colon 1\in A \mbox{ and } A\cap B\neq \emptyset \mbox{ for each }B\in \hb\right\}.
\]
Set $\hg(n,k)=\ha\cup \hb$.
\end{example}

It is easy to verify that for $n\geq 2k$, $\hg(n,k)$ is an intersecting $k$-graph with $\tau(\hg(n,k))=3$. Moreover,
\begin{align}\label{ineq-hgnk}
|\hg(n,k)| = \binom{n-1}{k-1}-\binom{n-k}{k-1}-\binom{n-k-1}{k-1}+\binom{n-2k}{k-1}+\binom{n-k-2}{k-3}+3.
\end{align}

The first author proved in \cite{F80} that for $k\geq 4$ and $n>n_0(k)$,
\begin{align}\label{ineq-frankl}
m(n,k,3)= |\hg(n,k)|.
\end{align}

Very recently, we proved the following:

\begin{thm}[\cite{FW2022}]
For $k\geq 7$ and $n\geq 2k$, $m(n,k,3)= |\hg(n,k)|$.
\end{thm}
This  leaves the cases $k=4,5,6$ open.

Let us mention that Kupavskii \cite{K2024} proved the same result for  $k\geq 100$ and $n\geq 2k$ with a shorter proof.

In this note, we settle the remaining cases.

\begin{thm}\label{thm-main}
For $k=4, 5,6$ and $n\geq 2k$, $m(n,k,3)= |\hg(n,k)|$.
\end{thm}

For $n=2k$, $|\hg(n,k)|=\frac{1}{2}\binom{2k}{k}=\binom{2k-1}{k-1}$. Thus $|\hf|\leq |\hg(n,k)|$ is true by the Erd\H{o}s-Ko-Rado Theorem. In the sequel we always assume $n>2k$.

\section{Preliminaries}

In this section, we recall some results in \cite{FW25} and a version of the Kruskal-Katona theorem.

\begin{prop}[\cite{FW25}]\label{prop-1}
Let $n> 2k$ and let $\hf\subset \binom{[n]}{k}$ be an intersecting family with covering number  at least 3. Then there exists an intersecting family $\hf'\subset \binom{[n]}{k}$ with $|\hf'|\geq |\hf|$ and $\tau(\hf')=3$.
\end{prop}

In view of above, in proving Theorem \ref{thm-main} we may assume $\hht^{(3)}(\hf)\neq \emptyset$. Define two 3-graphs $\hs$ and $\hr$ as:
\[
\hs=\{\{1,2,3\},\{1,4,5\},\{2,4,6\}\},\ \hr=\{\{1,2,3\},\{1,4,5\},\{2,3,5\}\}.
\]
Let $\hk_3(4)$ denote the complete 3-graph on 4 vertices.

For $\hf\subset \binom{[n]}{k}$ and $i\in [n]$, define
\[
\hf(i) =\{F\setminus \{i\}\colon i\in F\in \hf\},\ \hf(\bar{i}) =\{F\in \hf\colon i\notin F\}.
\]

\begin{prop}\label{prop-3}
Let $\hf\subset \binom{[n]}{k}$ be a saturated intersecting family with covering number 3. If $\hht^{(3)}(\hf)$ is non-trivial, then
 $\hht^{(3)}(\hf)=\hk_3(4)$ or it  contains  an isomorphic copy of $\hs$ or $\hr$.
\end{prop}

\begin{proof}
Let $\hht=\hht^{(3)}(\hf)$. If there exist $T,T'\in \hht$ with $|T\cap T'|=1$ then assume by symmetry $T=\{1,2,3\}$ and $T'=\{1,4,5\}$. Now choosing an arbitrary $T''\in \hht(\bar{1})$, $(T,T',T'')$ is isomorphic either to $\hs$ or to $\hr$.

Finally assume that $\hht$ is 2-intersecting, that is, $|T\cap T'|=2$ for all distinct $T,T'\in \hht$. By symmetry let $T=\{1,2,3\}$, $T'=\{1,2,4\}$. Then the 2-intersecting property implies $\hht(\bar{1})=\{\{2,3,4\}\}$, $\hht(\bar{2})=\{\{1,3,4\}\}$ and $\hht=\binom{[4]}{3}$ follows.
\end{proof}

Let us show the $\hht^{(3)}(\hf)=\hk_3(4)$ implies $|\hf|<|\hg(n,k)|$.

\begin{prop}\label{prop-4}
Let $\hf\subset \binom{[n]}{k}$ be a saturated intersecting family with covering number 3.
If $\hht^{(3)}(\hf)=\hk_3(4)$ and $n>2k\geq 6$,  then $|\hf|<|\hg(n,k)|$.
\end{prop}
\begin{proof}
Without loss of generality assume $\hht=\binom{[4]}{3}$.  Then $|F\cap[4]|\geq 2$ for all $F\in \hf$. Define
\[
\hf_i=\{F\in \hf\colon |F\cap [4]|=i\}.
\]
Let $P$, $P'\in \binom{[4]}{2}$ with $P\cup P' =[4]$. If $\hf(P,[4])=\emptyset$, then $P'$ is a cover of $\hf$, contradicting $\tau(\hf)\geq 3$.
Thus  both $\hf(P,[4])$ and $\hf(P',[4])$ are non-empty. By \eqref{ft92},
\[
|\hf(P,[4])|+|\hf(P',[4])|\leq \binom{n-4}{k-2}-\binom{n-k-2}{k-2}+1.
\]
Since $\binom{[4]}{2}$ can be partitioned into 3 disjoint pairs,
\begin{align*}
|\hf|=|\hf_2|+|\hf_3|+|\hf_4|&\leq  3\left(\binom{n-4}{k-2}-\binom{n-k-2}{k-2}+1\right)+4\binom{n-4}{k-3}+\binom{n-4}{k-4}.
\end{align*}
By Lemma 5.2 in \cite{FW25}, the RHS is less than $|\hg(n,k)|$.
\end{proof}

 The case $\hht^{(3)}(\hf)$ is a star was settled in \cite{FW25}.

 \begin{thm}[\cite{FW25}]\label{thm-3}
Let $\hf\subset \binom{[n]}{k}$ be an intersecting family with covering number 3. If $\hht^{(3)}(\hf)$ is a star, then for $n\geq 2k\geq 8$,
\[
|\hf| \leq |\hg(n,k)|.
\]
\end{thm}

It should be mentioned that the $k\geq 5$ case of Theorem \ref{thm-3} was deduced from Propositions 5.3 and 5.4 in \cite{FW25}. In the proofs of Propositions 5.3 and 5.4, inequalities (4.15) and (4.16) in \cite{FW25} were used. This is the only places where $k\geq 5$ is needed. However, it can be checked directly that (4.15) and (4.16) hold for $k=4$ as well. Thus Theorem \ref{thm-3} holds for $k=4$ as well.

Based on Propositions \ref{prop-1}, \ref{prop-3}, \ref{prop-4} and Theorem \ref{thm-3}, to prove Theorem \ref{thm-main} we are left with the case that $\hht^{(3)}(\hf)$ contains a copy of $\hs$ or $\hr$.

We need the following reformulation of the Kruskal--Katona Theorem, due to Hilton  \cite{Hilton}. To state it let us recall the definition of the lexicographic order on $\binom{[n]}{k}$. For two distinct sets $F,G\in \binom{[n]}{k}$ we say that $F$ {\it precedes} $G$ if
\[
\min\{i\colon i\in F\setminus G\}<\min\{i\colon i\in G\setminus F\}.
\]
Let $\hl(n,k,m)$ denote the family of the  first $m$ members of $\binom{[n]}{k}$ in the lexicographic order.

Two families $\ha\subset \binom{[n]}{a}$, $\hb\subset \binom{[n]}{b}$ are called {\it cross-intersecting} if $A\cap B\neq \emptyset$ for all $A\in \ha$, $B\in \hb$. We say that $\ha,\hb$ form {\it a saturated pair} if adding any new $a$-set to $\ha$ or any new $b$-set to $\hb$ would destroy the cross-intersecting property.

\begin{lem}[\cite{Kruskal, Katona66, Hilton}]\label{lem-hilton}
 Let $n,a,b$ be positive integers, $n\geq a+b$. Suppose that $\ha\subset \binom{[n]}{a}$ and $\hb\subset \binom{[n]}{b}$ are cross-intersecting. Then $\hl(n,a,|\ha|)$ and $\hl(n,b,|\hb|)$ are cross-intersecting as well.
\end{lem}

\begin{cor}\label{cor-hilton}
Let $\ha\subset \binom{[m]}{a}$, $\hb\subset \binom{[m]}{b}$ be cross-intersecting, $m>a+b$, $a>b$. If $|\hb|\geq \binom{m-1}{b-1}$ or $|\ha|\leq \binom{m-1}{a-1}$, then
\[
|\ha|+|\hb| \leq \binom{m-1}{a-1}+\binom{m-1}{b-1}.
\]
\end{cor}

\begin{proof}
By Lemma \ref{lem-hilton}, we may assume that $\ha=\hl(m,a,|\ha|)$, $\hb=\hl(m,b,|\hb|)$ and $\ha,\hb$ form a saturated pair. If $|\hb|\geq \binom{m-1}{b-1}$, then
\[
\left\{B\in \binom{[m]}{b}\colon 1\in B\right\} \subset  \hb.
\]
If $|\ha|\leq \binom{m-1}{a-1}$, then $1\in A$ for all $A\in \ha$. By saturatedness, $\{B\in \binom{[m]}{b}\colon 1\in B\} \subset  \hb$.
By the cross-intersecting property, we infer that
\[
|\ha|+|\hb|=\binom{m-1}{b-1}+|\ha(1)|+|\hb(\bar{1})|.
\]
Since $\ha(1)$ and $\hb(\bar{1})$ are cross-intersecting and $a-1\geq b$, we infer from Lemma \ref{lem-3.5} below that
\[
|\ha(1)|+|\hb(\bar{1})| \leq \binom{m-1}{a-1}.
\]
Thus the corollary follows.
\end{proof}

\section{Some inequalities}

We need the Frankl-Tokushige inequality \cite{FT92, F2022}.

\begin{thm}[\cite{FT92, F2022}]\label{thm-ft92}
Let $\ha \subset\binom{X}{a}$ and $\hb \subset \binom{X}{b}$ be non-empty cross-intersecting families with $n=|X|\geq a+b$, $a\leq b$. Then
\begin{align}\label{ft92}
|\ha|+|\hb|\leq \binom{n}{b} - \binom{n-a}{b}+1.
\end{align}
Moreover, unless $n=a+b$ or $a=b=2$ the inequality is strict for $|\ha|> 1$ and $|\hb|> 1$.
\end{thm}

For $\hf\subset \binom{[n]}{k}$ and $S\subset U\subset [n]$,  let
\[
\hf(S,U) =\{F\setminus U\colon F\in \hf,\ F\cap U=S\}.
\]
Define $f_S:=|\hf(S,U)|$.

\begin{prop}\label{prop-3.1}
Let $\hf\subset \binom{[n]}{k}$ be an intersecting family with $\tau(\hf)\geq 3$ and  let $U\subset [n]$ satisfy $|F\cap U|\geq 2$ for all $F\in \hf$.
Then for any $P\in \binom{U}{2}$,
\begin{align}\label{ineq-3.1}
f_{P} \leq \binom{n-|U|}{k-2}-\binom{n-k-|U|+2}{k-2}
\end{align}
For any $P,P'\in \binom{U}{2}$ with $P\cap P'=\emptyset$ and $n\geq 2k+|U|-4$,
\begin{align}\label{ineq-3.2}
  f_{P}+f_{P'} \leq \binom{n-|U|}{k-2}-\binom{n-k-|U|+2}{k-2}+1.
\end{align}
\end{prop}

\begin{proof}
By  $\tau(\hf)\geq 3$ there exists $F_0\in \hf$ such that $F_0\cap P=\emptyset$. Then $Q\cap (F_0\setminus U)\neq \emptyset$ for all $Q\in \hf(P,U)$. Since $|F_0\cap U|\geq 2$, we have
 \[
 f_{P}\leq \binom{n-|U|}{k-2}-\binom{n-|U|-(k-2)}{k-2}= \binom{n-|U|}{k-2}-\binom{n-k-|U|+2}{k-2}.
 \]

Note that $\hf(P,U)$ and  $\hf(P',U)$ are cross-intersecting and $|[n]\setminus U|=n-|U|\geq 2(k-2)$. If one of them is empty, then \eqref{ineq-3.2} follows from \eqref{ineq-3.1}.
 If both are non-empty, then by \eqref{ft92}
 \[
  f_{P}+f_{P'} \leq \binom{n-|U|}{k-2}-\binom{n-k-|U|+2}{k-2}+1.
 \]
\end{proof}

\begin{prop}
Let $\hf\subset \binom{[n]}{k}$ be an intersecting family with $\tau(\hf)\geq 3$. Let $U\subset [n]$ satisfy $|U|\in \{5,6\}$ and $|F\cap U|\geq 2$ for all $F\in \hf$. Then for any $P,P'\in \binom{U}{2}$ with $P\cap P'=\emptyset$ and $n\geq 2k+|U|-4$,
\begin{align}\label{ineq-key}
  f_{P}+f_{P'}+f_{U\setminus P}+f_{U\setminus P'} \leq \binom{n-|U|}{k-2}&-\binom{n-k-|U|+2}{k-2} \nonumber\\[3pt]
  & +\binom{n-|U|}{k-|U|+2} +\binom{n-|U|-1}{k-|U|+1}.
\end{align}
\end{prop}

\begin{proof}
Note that if $ f_{U\setminus P} \geq \binom{n-|U|-1}{k-|U|+1}$ then by Corollary \ref{cor-hilton},
 \[
 f_{P}+ f_{U\setminus P} \leq \binom{n-|U|-1}{k-3}+\binom{n-|U|-1}{k-|U|+1}.
 \]
 Consequently, if $f_{U\setminus P}\geq \binom{n-|U|-1}{k-|U|+1}$ and $f_{U\setminus P'}\geq \binom{n-|U|-1}{k-|U|+1}$, then
 \[
  f_{P}+f_{P'}+ f_{U\setminus P}+f_{U\setminus P'} \leq 2\binom{n-|U|-1}{k-3}+2\binom{n-|U|-1}{k-|U|+1}.
 \]
 We have to show that
 \begin{align*}
 \binom{n-|U|}{k-2}-\binom{n-k-|U|+2}{k-2}+&\binom{n-|U|}{k-|U|+2} +\binom{n-|U|-1}{k-|U|+1} \\[3pt]
 & \geq  2\binom{n-|U|-1}{k-3}+2\binom{n-|U|-1}{k-|U|+1}.
 \end{align*}
For $k\geq 4$, $\binom{n-|U|}{k-2}-\binom{n-k-|U|+2}{k-2}\geq  \binom{n-|U|-1}{k-3}+\binom{n-|U|-2}{k-3}$. Hence it suffices to prove
\[
\binom{n-|U|}{k-|U|+2}+\binom{n-|U|-2}{k-3}\geq \binom{n-|U|-1}{k-3}+\binom{n-|U|-1}{k-|U|+1}.
\]
Equivalently $\binom{n-|U|-1}{k-|U|+2}\geq \binom{n-|U|-2}{k-4}$, which is true for $|U|=5$ or 6.

 If one of $f_{U\setminus P}$, $f_{U\setminus P'}$ is smaller than $\binom{n-|U|-1}{k-|U|+1}$, then
 \begin{align}\label{ineq-4.2}
 f_{U\setminus P}+f_{U\setminus P'} \leq \binom{n-|U|}{k-|U|+2}+\binom{n-|U|-1}{k-|U|+1}-1.
 \end{align}
By Proposition  \ref{prop-3.1},
 \begin{align}\label{ineq-4.3}
 f_{P}+f_{P'} \leq \binom{n-|U|}{k-2}-\binom{n-k-|U|+2}{k-2}+1.
\end{align}
Adding \eqref{ineq-4.2} and \eqref{ineq-4.3},  the claim follows.
\end{proof}


For $k=4$ and $|U|=5$, we  need a slightly better upper bound.

\begin{prop}\label{prop-3.4}
Let $\hf\subset \binom{[n]}{4}$ be an intersecting family with $\tau(\hf)\geq 3$ and $n\geq 9$. Suppose that  $|F\cap [5]|\geq 2$ for all $F\in \hf$. Then for any $P,P'\in \binom{[5]}{2}$ with $P\cap P'=\emptyset$,
\begin{align}\label{ineq-key2}
f_P+f_{P'}+f_{[5]\setminus P}+f_{[5]\setminus P'} \leq 3(n-6),
\end{align}
where equality holds if and only if $f_P=0$, $f_{P'}=2n-13$ or $f_{P'}=0$,  $f_{P}=2n-13$.
\end{prop}

\begin{proof}
If $f_{[5]\setminus P},f_{[5]\setminus P'}\geq 1$, then by Corollary \ref{cor-hilton},
\begin{align*}
f_P+f_{[5]\setminus P}\leq 1+(n-6)=n-5,\ f_{P'}+f_{[5]\setminus P'}\leq 1+(n-6)=n-5.
\end{align*}
It follows that
  \[
f_P+f_{P'}+f_{[5]\setminus P}+f_{[5]\setminus P'} \leq 2(n-5).
  \]
Thus we may assume that one of $f_{[5]\setminus P}$, $f_{[5]\setminus P'}$ is zero. Without loss of generality assume $f_{[5]\setminus P'}=0$.

 By \eqref{ineq-3.2}, we infer that
  \begin{align*}
f_P+f_{P'} \leq \binom{n-5}{2}-\binom{n-7}{2}+1=2(n-6).
 \end{align*}
 Thus for  $f_{[5]\setminus P}=0$,
 \begin{align*}
f_P+f_{P'}+f_{[5]\setminus P}+f_{[5]\setminus P'}=f_P+f_{P'} \leq 2(n-6).
 \end{align*}

If $f_{[5]\setminus P}\geq 1$, then by Corollary \ref{cor-hilton}, $f_P+f_{[5]\setminus P}\leq n-5$. By \eqref{ineq-3.1}, $f_{P'}\leq (n-6)+(n-7)$. Thus
 \begin{align*}
 f_{P}+f_{P'}+f_{[5]\setminus P}+f_{[5]\setminus P'}&= (f_P+f_{[5]\setminus P})+f_{P'}\\[3pt]
 &\leq (n-5)+(n-6)+(n-7)=3(n-6).
 \end{align*}
If equality holds, then $f_{P'}= (n-6)+(n-7)$. It follows that $f_P\leq 1$. If $f_P=1$ then $f_{[5]\setminus P}\leq 2$. Thus $f_P=0$ and $f_{[5]\setminus P}=n-5$.
\end{proof}

The next statement can be deduced using an old argument of Sperner \cite{Sperner}.

\begin{lem}[\cite{F2022,Sperner}]\label{lem-3.5}
Let $\hf\subset \binom{[n]}{k}$ be an intersecting family  and let $U\subset [n]$. If $A,B\subset U$, $A\cap B=\emptyset$ then for $n\geq 2k-|A|-|B|+|U|$,
\begin{align}\label{ineq-sperner}
\alpha(A)+\alpha(B)\leq 1,
\end{align}
where $\alpha(A)=f_A/\binom{n-|U|}{k-|A|}$ and  $\alpha(B)=f_B/\binom{n-|U|}{k-|B|}$.
\end{lem}

\begin{prop}
Let $\hf\subset \binom{[n]}{k}$ be an intersecting family with $\tau(\hf)\geq 3$. Suppose that  $|F\cap [5]|\geq 2$ for all $F\in \hf$ and $n> 2k$. Then for any $P,P'\in \binom{[5]}{2}$ with $P\cap P'=\emptyset$,
\begin{align}\label{ineq-3.3}
f_P+f_{P'}+f_{[5]\setminus P}+f_{[5]\setminus P'} \leq \binom{n-5}{k-2}+\binom{n-5}{k-3}.
\end{align}
\end{prop}
\begin{proof}
Let $\alpha(P)=f_P/\binom{n-5}{k-2}$ and $\alpha([5]\setminus P)=f_{[5]\setminus P}/\binom{n-5}{k-3}$. Then by the intersecting property and \eqref{ineq-sperner},
\[
\alpha(P)+\alpha(P')\leq 1,\ \alpha(P)+\alpha([5]\setminus P)\leq 1,\ \alpha(P')+\alpha([5]\setminus P')\leq 1.
\]
It follows that
\begin{align*}
&\quad \ f_P+f_{P'}+f_{[5]\setminus P}+f_{[5]\setminus P'}\\[3pt]
& = (\alpha(P)+\alpha(P')) \binom{n-5}{k-2}+(\alpha([5]\setminus P)+\alpha([5]\setminus P')) \binom{n-5}{k-3}\\[3pt]
& \leq (\alpha(P)+\alpha(P')) \binom{n-5}{k-2}+(2-\alpha(P)-\alpha(P')) \binom{n-5}{k-3}\\[3pt]
&=2 \binom{n-5}{k-3}+(\alpha(P)+\alpha(P')) \left(\binom{n-5}{k-2}-\binom{n-5}{k-3}\right)\\[5pt]
&\leq \binom{n-5}{k-2}+\binom{n-5}{k-3}.
\end{align*}
\end{proof}

\section{The case $n=2k+1$ and $k\geq 5$}

By carefully rechecking the proofs in \cite{FW25} we discovered that the proof of $|\hf|\leq |\hg(n,k)|$ is missing in the case that $n=2k+1$ and $\hht^{(3)}(\hf)$ contains a copy of $\hs$.

The next proposition fills in this gap.

\begin{prop}\label{prop-2k+1}
Let $\hf\subset \binom{[n]}{k}$ be an intersecting family with $\tau(\hf)=3$  If $\hht^{(3)}(\hf)$ is non-trivial, then $|\hf|< |\hg(n,k)|$ for $n=2k+1$ and $k\geq 5$.
\end{prop}

Let $\Delta(\hf)$ denote the maximum degree of $\hf$, that is, $\Delta(\hf)=\max_{1\leq i\leq n} |\hf(i)|$. For the proof of Proposition \ref{prop-2k+1}, we need the following result of the first author.

\begin{thm}[\cite{F87-2}]\label{thm-F87}
Let $2\leq \ell\leq k$, $n>2k$. Suppose that $\hf\subset \binom{[n]}{k}$ is an intersecting family satisfying
\begin{align}
\Delta(\hf) \leq \binom{n-1}{k-1} -\binom{n-\ell-1}{k-1}.
\end{align}
Then
\begin{align}
|\hf| \leq \binom{n-1}{k-1}-\binom{n-\ell-1}{k-1}+\binom{n-\ell-1}{k-\ell}.
\end{align}
\end{thm}

\begin{proof}[Proof of Proposition \ref{prop-2k+1}]
Since $\hht^{(3)}(\hf)$ is non-trivial, we infer that $\Delta(\hf) \leq \binom{n-1}{k-1} -\binom{n-4}{k-1}$. Then Theorem \ref{thm-F87} implies
\[
|\hf| \leq \binom{n-1}{k-1}-\binom{n-4}{k-1}+\binom{n-4}{k-3}.
\]
Since for $n=2k+1$,
\[
|\hg(n,k)| = \binom{n-1}{k-1}-\binom{k+1}{k-1}-\binom{k}{k-1}+\binom{1}{k-1}+\binom{k-1}{k-3}+3=\binom{n-1}{k-1}-3k+4
\]
and
\[
\binom{n-1}{k-1}-\binom{n-4}{k-1}+\binom{n-4}{k-3}=\binom{n-1}{k-1} -\binom{2k-3}{k-1}+\binom{2k-3}{k-3},
\]
we are left to show that
\[
f(k):=\binom{2k-3}{k-1}-\binom{2k-3}{k-3} -3k+4 \geq 0.
\]
By direct computation we have $f(5)=3>0$. Moreover,
\begin{align*}
f(k+1) &=\binom{2k-1}{k}-\binom{2k-1}{k-2} -3k+1\\[3pt]
&=\binom{2k-3}{k}+2\binom{2k-3}{k-1}+\binom{2k-3}{k-2}-\left(\binom{2k-3}{k-2}+2\binom{2k-3}{k-3}+\binom{2k-3}{k-4}\right)\\[3pt]
&\qquad \qquad -3k+1\\[3pt]
&=2\binom{2k-3}{k-1}-\binom{2k-3}{k-3}-\binom{2k-3}{k-4}-3k+1\\[3pt]
&=f(k)+\binom{2k-3}{k-1}-\binom{2k-3}{k-4}-3.
\end{align*}
Then for $k\geq 5$,
\[
f(k+1) -f(k) = \binom{2k-3}{k-1}-\binom{2k-3}{k-4}-3= \frac{6(k-1)}{k(k+1)}\binom{2k-3}{k-1} -3>0.
\]
Thus the proposition follows.
\end{proof}

\section{Proof of Theorem \ref{thm-main}}

Let $\hf\subset \binom{[n]}{k}$ be a saturated intersecting family with $\tau(\hf)\geq 3$.
 By Propositions \ref{prop-1}, \ref{prop-3}, \ref{prop-4} and Theorems \ref{thm-3}, we may assume that $\hht^{(3)}(\hf)$  contains either an isomorphic copy of $\hs$ or an isomorphic copy of $\hr$.

Let $\hp(\hr)$ and $\hp(\hs)$ be the family of 2-covers of $\hr$ and $\hs$, respectively.

\begin{claim}
If $\hr \subset \hht^{(3)}(\hf)$, then $|F\cap [5]|\geq 3$ or $F\cap [5]\in \hp(\hr)$ for all $F\in \hf$. Similarly, if $\hs \subset \hht^{(3)}(\hf)$, then $|F\cap [6]|\geq 3$ or $F\cap [6]\in \hp(\hs)$ for all $F\in \hf$.
\end{claim}

\begin{proof}
Since $\hr$ is non-trivial, $|F\cap [5]|\geq 2$ for all $F\in \hf$. If $|F\cap [5]|=2$ then $F\cap [5]$ has to be a 2-cover of $\hr$.
\end{proof}

Now we distinguish two cases.

\vspace{3pt}
{\bf Case 1.} $\hht^{(3)}(\hf)$ contains a copy of $\hr$.
\vspace{3pt}

Without loss of generality assume $\hr\subset \hht^{(3)}(\hf)$.
Note that $|[5]|=3+2$ implies $\hp(\hr)=\binom{[5]}{2}\setminus \{[5]\setminus R\colon R\in \hr\}$. That is,
\[
\hp(\hr)=\left\{(1,2),(1,3),(1,5),(2,4),(2,5),(3,4),(3,5)\right\}.
\]
Let us define an auxiliary graph $G$ with the vertex set $\hp(\hr)$ and the edge set
\[
\{(P,P')\colon P,P'\in \hp(\hr),\ P\cap P'=\emptyset \}.
\]
\begin{figure}[H]
\centering
\includegraphics[width=0.3\textwidth]{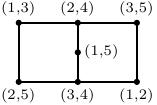}\\[3pt]
\caption{The auxiliary graph $G$.}
\end{figure}
Let
\[
\hi=\left\{P\in \hp(\hr)\colon |\hf(P,[5])| > \binom{n-6}{k-3} \right\}.
\]
\begin{claim}\label{claim-6}
 $G$ can be partitioned into 3 pairwise disjoint edges and a vertex $P_0$ that is not in $\hi$.
\end{claim}
\begin{proof}
Since $\hf(P,[5])$, $\hf(P',[5])$ are cross-intersecting for any $P,P'\in \hp(\hr)$ with $P\cap P'=\emptyset$, we infer that $\hi$ is an independent set of $G$.  Then there are two cases: (i) $(1,5) \notin \hi$. Then $G- (1,5)$ is $C_6$ which can be partitioned into 3 edges. (ii) $(1,5) \in \hi$. Then $((1,3),(2,5))$, $((3,5),(1,2))$ and $((1,5),(2,4))$ are 3 pairwise disjoint edges in $G$ and $(3,4)\notin \hi$.  Thus, $G$ can be partitioned into 3 pairwise disjoint edges and a vertex $P_0$ that is not in $\hi$.
\end{proof}

Define $\hf_i=\{F\in \hf\colon |F\cap [5]|=i\}$, $i=2,3,\ldots,6$.
Since $P_0\notin \hi$, by Corollary \ref{cor-hilton} we infer that
\begin{align}\label{ineq-new1}
f_{P_0}+f_{[5]\setminus P_0}\leq \binom{n-6}{k-3}+\binom{n-6}{k-4}= \binom{n-5}{k-3}.
\end{align}
By \eqref{ineq-key}, for any $P, P'\in \hp(\hr)$ with $P\cap P'=\emptyset$ we have
\begin{align}\label{ineq-4.4}
f_P+f_{P'}+f_{[5]\setminus P}+f_{[5]\setminus P'} &\leq \binom{n-5}{k-2}-\binom{n-k-3}{k-2}+\binom{n-5}{k-3}+\binom{n-6}{k-4}\nonumber\\[3pt]
&= \binom{n-4}{k-2}-\binom{n-k-3}{k-2}+\binom{n-6}{k-4}.
\end{align}
By Claim \ref{claim-6}, \eqref{ineq-new1} and \eqref{ineq-4.4}
\begin{align*}
|\hf_2|+|\hf_3|&=3\left(\binom{n-4}{k-2}-\binom{n-k-3}{k-2}+\binom{n-6}{k-4}\right)+\binom{n-5}{k-3}+3\binom{n-5}{k-3}\\[3pt]
&=3\left(\binom{n-4}{k-2}-\binom{n-k-3}{k-2}\right)+3\binom{n-6}{k-4}+4\binom{n-5}{k-3}.
\end{align*}
Therefore we conclude that
 \begin{align*}
 |\hf| &= |\hf_2|+|\hf_3|+|\hf_4|+|\hf_5|\\[3pt]
 &\leq 3\left(\binom{n-4}{k-2}-\binom{n-k-3}{k-2}\right)+3\binom{n-6}{k-4}+4\binom{n-5}{k-3}
 +5\binom{n-5}{k-4}+\binom{n-5}{k-5}\\[3pt]
 &= 3\left(\binom{n-4}{k-2}-\binom{n-k-3}{k-2}\right)+4\binom{n-4}{k-3} +\binom{n-5}{k-4}+3\binom{n-6}{k-4}+\binom{n-5}{k-5}.
 \end{align*}
  For $k=5$ and $n\geq 13$ the above inequality yields,
 \[
 |\hf| \leq  \frac{1}{2}(16 n^2 - 196 n +636) <  \frac{1}{2}(21 n^2 - 295 n +1102)=|\hg(n,5)|.
 \]
 For $k=6$ and $n\geq 14$, we obtain
 \begin{align*}
 |\hf| &\leq  \frac{1}{6}(19 n^3- 408 n^2 + 3107 n -8322)\\[3pt]
  &<  \frac{1}{6}( 31 n^3 - 792 n^2+ 7157 n-22632)\\[3pt]
  &=|\hg(n,6)|.
 \end{align*}

 For $k=4$ and $n\geq 9$, by \eqref{ineq-new1} we have
\begin{align}\label{ineq-new11}
f_{P_0}+f_{[5]\setminus P_0}\leq \binom{n-5}{4-3}= n-5.
\end{align}
By Claim \ref{claim-6}, \eqref{ineq-new11} and \eqref{ineq-key2},
\begin{align*}
|\hf_2|+|\hf_3|&=3\times 3(n-6)+(n-5)+\left(\binom{5}{3}-6-1\right)(n-5)=13n-74.
\end{align*}
Therefore, we conclude that
 \begin{align}\label{ineq-4.11}
 |\hf| &= |\hf_2|+|\hf_3|+|\hf_4|\leq 13n-74+\binom{5}{4}=13n-69=|\hg(n,4)|.
 \end{align}
 By Proposition \ref{prop-3.4},  equality holds in \eqref{ineq-4.11} if and only if equality holds in \eqref{ineq-new11} and $f_P=2n-13$, $f_{P'}=0$ or $f_P=0$, $f_{P'}=2n-13$ for each pair $(P,P')$. If $f_P=2n-13$ then by $\tau(\hf)=3$ there exists $F_0\in \hf$ such that $P\cap F_0=\emptyset$ and  $F_0\cap [5]=2$. Then $P'':=F_0\cap [5]\in \hp(\hr)$. By \eqref{ineq-3.2}, $f_P=2n-13$ implies $f_{P''}=1$. Then $P''$ has to be $P_0$. However, $f_{P_0}=1$ would imply $f_{P_0}+f_{[5]\setminus P_0}\leq 3$, contradicting the fact equality holds in \eqref{ineq-new11}.  Thus there is no equality in \eqref{ineq-4.11}.

For the cases $k=5$, $11
\leq n\leq 12$ and $k=6$, $n=13$, we need to estimate $|\hf|$ in a different way.
By Claim \ref{claim-6}, \eqref{ineq-3.3} and \eqref{ineq-new1}, we infer that
\begin{align*}
|\hf_2|+|\hf_3|&=3\left(\binom{n-5}{k-2}+\binom{n-5}{k-3}\right)+\binom{n-5}{k-3}+3\binom{n-5}{k-3}.
\end{align*}
Thus,
\[
  |\hf| \leq 3\binom{n-5}{k-2}+7\binom{n-5}{k-3}+5\binom{n-5}{k-4}+\binom{n-5}{k-5}.
 \]
 Now for $k=5$ and $11\leq n\leq 12$,
 \[
 |\hf| \leq  \frac{1}{2}(n^3- 11 n^2+ 40 n-48) <  \frac{1}{2}(21 n^2 - 295 n +1102)=|\hg(n,5)|.
 \]
  For $k=6$ and $n=13$,
 \begin{align*}
 |\hf| &\leq  \frac{1}{24}(3 n^4- 50 n^3+ 309 n^2 - 838 n+ 840)\\[3pt]
 &<  \frac{1}{6}( 31 n^3 - 792 n^2+ 7157 n-22632)\\[3pt]
 &=|\hg(n,6)|.
 \end{align*}

\vspace{3pt}
{\bf Case 2.} $\hht^{(3)}(\hf)$ contains a copy of $\hs$ but no copy of $\hr$.
\vspace{3pt}

Without loss of generality assume $\hs\subset \hht^{(3)}(\hf)$. By Proposition \ref{prop-2k+1}, we may assume $n\geq 2k+2$ for $k\geq 5$ in this case.  Note that
\[
\hp(\hs) =\{(1,2),(3,4)\} \cup \{(2,4),(1,6)\}\cup \{(1,4),(2,5)\},
\]
i.e., it can be partitioned into three disjoint pairs. Let $P$, $P'$ be one of the pairs. Then by $n\geq 2k+2$ and \eqref{ineq-key},
\begin{align}\label{ineq-2.2}
f_P+f_{P'}+ f_{[6]\setminus P}+f_{[6]\setminus P'}\leq\binom{n-6}{k-2}-\binom{n-k-4}{k-2}+\binom{n-6}{k-4}+\binom{n-7}{k-5}.
\end{align}
It follows that
 \begin{align}\label{ineq-4.5}
 |\hf_2|+|\hf_4| &\leq 3\left(\binom{n-6}{k-2}-\binom{n-k-4}{k-2}+\binom{n-6}{k-4}+\binom{n-7}{k-5}\right)+9\binom{n-6}{k-4}\nonumber\\[3pt]
 &\leq 3\left(\binom{n-6}{k-2}-\binom{n-k-4}{k-2}\right)+12\binom{n-6}{k-4}+3\binom{n-7}{k-5}.
 \end{align}

 For any $T\in \binom{[6]}{3}$, by \eqref{ineq-sperner} we have
 \[
f_T+f_{[6]\setminus T} \leq \binom{n-6}{k-3}.
 \]
 It follows that $|\hf_3|\leq \frac{1}{2}\binom{6}{3}\binom{n-6}{k-3}=10\binom{n-6}{k-3}$. Thus,
 \begin{align*}
 |\hf| &=|\hf_2|+|\hf_3|+|\hf_4|+|\hf_5|+|\hf_6|\\[3pt]
 & \leq 3\left(\binom{n-6}{k-2}-\binom{n-k-4}{k-2}\right)+10\binom{n-6}{k-3}
 +12\binom{n-6}{k-4}+3\binom{n-7}{k-5}\\[3pt]
 &\qquad\qquad+6\binom{n-6}{k-5}+\binom{n-6}{k-6}\\[3pt]
 & = 3\left(\binom{n-6}{k-2}-\binom{n-k-4}{k-2}\right)+\binom{n-3}{k-3}
 +3\binom{n-4}{k-3}+6\binom{n-5}{k-3}-3\binom{n-7}{k-4}.
 \end{align*}

 For $k=5$ and $n\geq 12$,
 \[
 |\hf| \leq  \frac{1}{2}(19 n^2- 259 n +948) <  \frac{1}{2}(21 n^2 - 295 n +1102)=|\hg(n,5)|.
 \]
 For $k=6$ and $n\geq 14$,
 \begin{align*}
 |\hf| &\leq  \frac{1}{6}(22 n^3- 516 n^2+ 4328 n-12786)\\[3pt]
       &<  \frac{1}{6}( 31 n^3 - 792 n^2+ 7157 n-22632)\\[3pt]
       &=|\hg(n,6)|.
 \end{align*}

We are left with the $k=4$ case. Note that if $T \notin \hht^{(3)}(\hf)$ for some $T\in \binom{[6]}{3}$, then there exists $F_0\in \hf$ such that $F_0\cap T=\emptyset$. As $|F_0\cap [6]|\geq 2$, it follows that $f_T\leq 2$.
If $T$, $[6]\setminus T\notin \hht^{(3)}(\hf)$, then by the cross-intersecting property $f_T+f_{[6]\setminus T}\leq 2$. Let $\hht=\hht^{(3)}(\hf)\cap \binom{[6]}{3}$. Then
\begin{align}\label{ineq-4.1}
|\hf_3|\leq |\hht|(n-6)+\left(\frac{1}{2}\binom{6}{3}-|\hht|\right)2 =|\hht|(n-8)+20.
\end{align}

\begin{claim}\label{claim-5}
$|\hht| \leq 4$.
\end{claim}

\begin{proof}
Since $\hf$ is saturated, by Proposition \ref{prop-1.4}, $\hht$ is intersecting. Then $|\hht\cap\{T,[6]\setminus T\}|\leq 1$ for all $T\in \binom{[6]}{3}$. Since $\hht$ contains no copy of $\hr$, no member of the following 6 pairs is in $\hht$:
\begin{align*}
&(\{2,3,4\},\{1,5,6\}), (\{2,3,5\},\{1,4,6\}), (\{2,4,5\},\{1,3,6\}),\\[3pt]
& (\{3,4,5\},\{1,2,6\}), (\{3,4,6\},\{1,2,5\}), (\{1,3,4\},\{2,5,6\}).
\end{align*}
E.g. if $\{2,3,4\}\in \hht$, then $\{1,2,3\}$, $\{2,3,4\}$ and $\{1,4,5\}$ form a copy of $\hr$; if $\{1,2,6\}\in \hht$ then $\{1,2,3\}$, $\{2,4,5\}$ and $\{1,4,5\}$ form a copy of $\hr$.
Thus $|\hht| \leq \frac{1}{2}\binom{6}{3}-6=4$.
\end{proof}

By \eqref{ineq-4.1} and Claim \ref{claim-5},  for $n\geq 9$ we have
\begin{align}\label{ineq-4.7}
|\hf_3|\leq |\hht|(n-8)+20\leq 4n -12.
\end{align}
If $n\geq 10$, by \eqref{ineq-4.5} we have
\begin{align}\label{ineq-4.6}
|\hf_2|+|\hf_4| \leq 6n-33.
\end{align}
Adding \eqref{ineq-4.6} and \eqref{ineq-4.7}, we conclude that
\[
|\hf|\leq 10n-45< 13n-69 =|\hg(n,4)|.
\]

Now assume  $n=9$ and $k=4$. If for each  $(P,P')$ of three disjoint pairs,
\[
f_P+f_{P'}+f_{[6]\setminus P}+f_{[6]\setminus P'}\leq 5,
\]
then
\[
|\hf_2|+|\hf_4| \leq 3\times 5+ \binom{6}{4}-6 =24.
\]
Together with \eqref{ineq-4.7}, we obtain that
\begin{align}\label{ineq-4.10}
|\hf|\leq 24+4\times 9 -12=48=|\hht(9,4)|.
 \end{align}
 If  equality holds in \eqref{ineq-4.10}, then
$f_P+f_{P'}=5$.  Since $\hf(P,[6]),\hf(P',[6])\subset \binom{\{7,8,9\}}{2}$, we have $f_P=3$, $f_{P'}=2$ or $f_P=2$, $f_{P'}=3$.  By the intersecting property,  $f_{P\cup \{x\}}+f_{P'\cup \{y\}} \leq 1$ and $f_{P\cup \{y\}}+f_{P'\cup \{x\}} \leq 1$, where $\{x,y\}=[6]\setminus (P\cup P')$. Then
\[
|\hf_3|\leq |\hht|(n-6)+\left(\frac{1}{2}\binom{6}{3}-|\hht|\right)2 -2\leq 22.
\]
Thus $|\hf|\leq 24+22=46$. Therefore there is no equality in \eqref{ineq-4.10}.

Note that $\hf(P,[6]),\hf(P',[6])\subset \binom{\{7,8,9\}}{2}$. It is easy to see that $f_P+f_{P'}+f_{[6]\setminus P}+f_{[6]\setminus P'}\leq 6$. Then
\begin{align}\label{ineq-4.8}
|\hf_2|+|\hf_4| \leq 3\times 6+ \binom{6}{4}-6 =27.
\end{align}
Suppose that $f_P+f_{P'}+f_{[6]\setminus P}+f_{[6]\setminus P'}=6$ for some $(P,P')$.  It follows that
$f_{[6]\setminus P}=f_{[6]\setminus P'}=0$ and $\hf(P,[6])=\hf(P',[6])= \binom{\{7,8,9\}}{2}$. Let $\{x,y\}=[6]\setminus (P\cup P')$. By the intersecting property,
\[
\hf(P\cup \{x\},[6])=\hf(P\cup \{y\},[6])=\hf(P'\cup \{x\},[6])=\hf(P'\cup \{y\},[6])=\emptyset.
\]
Then
\begin{align}\label{ineq-4.9}
|\hf_3|\leq |\hht|(n-6)+\left(\frac{1}{2}\binom{6}{3}-2-|\hht|\right)2\leq |\hht|(n-8)+16\leq 4n -16= 20.
\end{align}
Adding \eqref{ineq-4.6} and \eqref{ineq-4.7}, we conclude that $|\hf|\leq 27+20=47<|\hht(9,4)|$.

\section{Concluding remarks}

Combining the results of \cite{FW25} and the present paper,  for all $n\geq 2k\geq 6$,
\[
m(n,k,3) =\binom{n-1}{k-1}-\binom{n-k}{k-1}-\binom{n-k-1}{k-1}+\binom{n-2k}{k-1}+\binom{n-k-2}{k-3}+3.
\]
For $n=2k$ there are many intersecting families with covering number at least 3 and size $|\hg(n,k)|=\frac{1}{2}{2k \choose k}$. However, examining carefully the proofs one can show that for $n>2k\geq 8$, up to isomorphism
$\hg(n,k)$ is the only family attaining equality.

A complete solution to $m(n,k,s)$ for larger values of $s$ appears to be hopeless. For some partial results, we refer to \cite{FOT95}, \cite{FuruyaTakatou}, \cite{FW24} for  $m(n,k,4)$ and $m(n,k,5)$.

\vspace{8pt}
{\noindent \bf Acknowledgement.} The second author was supported by
National Natural Science Foundation of China Grant no. 12471316 and Natural Science Foundation of Shanxi Province Grant no. RD2500002993.

\end{document}